\documentclass{amsart}
\usepackage{graphicx, amssymb, eucal, hyperref, color, amsmath, amsthm, tikz}
\usetikzlibrary{cd}
\newtheorem{iThm}{Theorem}

\newtheorem{thm}{Theorem}[section]

\newtheorem{lem}[thm]{Lemma}
\newtheorem{fact}[thm]{Fact}

\newtheorem{defn}[thm]{Definition}

\theoremstyle{remark}
\newtheorem{rem}[thm]{Remark}
\newtheorem{ex}[thm]{Example}

\newtheorem{q}[thm]{Question}
\newcommand{\vertiii}[1]{{\left\vert\kern-0.25ex\left\vert\kern-0.25ex\left\vert #1 
    \right\vert\kern-0.25ex\right\vert\kern-0.25ex\right\vert}}
\newcommand{\norm}[1]{\left\Vert#1\right\Vert}
\newcommand{\abs}[1]{\left\vert#1\right\vert}

\newcommand{\op}[1]{\operatorname{#1}}
\newcommand{\ignore}[1]{}

\renewcommand{\phi}{\varphi}

\newcommand{\into}{\hookrightarrow}
\newcommand{\mult}{\operatorname{mult}}
\begin{document}
\title{Concrete barriers to quantifier elimination in finite dimensional C*-algebras}%
\author[C. J. Eagle]{Christopher J. Eagle} 
\address[C. J. Eagle]{University of Victoria, Department of Mathematics and Statistics, PO BOX 1700 STN CSC, Victoria, British Columbia, Canada, V8W 2Y2}%
\email{eaglec@uvic.ca}
\urladdr{http://www.math.uvic.ca/~eaglec}

\author[T. Schmid]{Todd Schmid}
\address[T. Schmid]{Department of Mathematics, University of Toronto, Toronto, Ontario, Canada,  M5S 2E4}%
\email{tschmid@math.toronto.edu}
\urladdr{https://www.math.toronto.edu/tschmid/}

%\begin{keyword}
%Omitting types, Baire category, type-space functor, Banach-Mazur game, co-analytic space
%
%\MSC[2010] 03C95 \sep 03C15 \sep 03C30 \sep 54E52 \sep 54G20 \sep 54H99 \sep 54A35
%\end{keyword}

\date{\today}%
% ----------------------------------------------------------------
\begin{abstract}
Work of Eagle, Farah, Goldbring, Kirchberg, and Vignati shows that the only separable C*-algebras that admit quantifier elimination in continuous logic are $\mathbb{C}$, $\mathbb{C}^2$, $M_2(\mathbb{C})$, and the continuous functions on the Cantor set.  We show that, among finite dimensional C*-algebras, quantifier elimination does hold if the language is expanded to include two new predicate symbols: One for minimal projections, and one for pairs of unitarily conjugate projections.  Both of these predicates are definable, but not quantifier-free definable, in the usual language of C*-algebras.  We also show that adding just the predicate for minimal projections is sufficient in the case of full matrix algebras, but that in general both new predicate symbols are required.
\end{abstract}
\maketitle
% ----------------------------------------------------------------
\section*{Introduction}\label{sec:Introduction}
In recent years there has been a significant interaction between model theory and operator algebras.  A key step in applying model-theoretic methods to a particular structure is to understand the \emph{definable} objects in that structure, that is, the objects that can be described by a formula in the formal language being used.  One measure of the complexity of a definition is the number of alternations of quantifiers needed; on this measure, the simplest definitions are the ones that require no quantifiers at all.  A theory is said to have \emph{quantifier elimination} if, in structures satisfying that theory, every definable object can be defined without quantifiers (a precise definition is given in Definition \ref{defn:QE}, below).

In the context of C*-algebras, quantifier elimination is a rare phenomenon.  The results of Farah, Kirchberg, Vignati, and the first author in \cite{Eagle2017}, and of Goldbring, Vignati, and the first author in \cite{Eagle2016}, show that the only separable C*-algebras that admit quantifier elimination in the natural language of C*-algebras are $\mathbb{C}$, $\mathbb{C}^2$, $M_2(\mathbb{C})$,  and $C($Cantor space$)$.  Thus in all other separable C*-algebras there are definable objects that cannot be defined without quantifiers.  The methods of proof used in \cite{Eagle2016} and \cite{Eagle2017} do not produce specific formulas whose quantifiers cannot be eliminated.  

Having a concrete description of which objects require quantifiers in their definitions can be very useful.  For instance, Tarski \cite{Tarski1951} showed that the theory of the real numbers does not eliminate quantifiers in the language of fields, but it does in the language of \emph{ordered} fields.  Since the ordering is definable in the language of fields, Tarski's result shows both that the ordering is not quantifier-free definable, and also that the ordering is, in a sense, the \emph{only} obstacle to quantifier elimination for the real field.  Tarski's theorem thus allows for a deep understanding of the definable objects in the real field; for instance, once quantifier elimination for the ordered field of real numbers is known, it is not difficult to show that this structure is $o$-minimal.  The fact that the ordered real field has quantifier elimination, and that the elimination of quantifiers can be carried out algorithmically, has also seen practical applications in computer algebra (using algorithms from \cite{Collins1975}).

In this paper we work in the context of finite dimensional C*-algebras.  We exhibit concrete definable objects which are the only barriers to quantifier elimination for finite dimensional C*-algebras, in the same sense that the ordering is the only barrier to quantifier elimination for the real field. Specifically, let $L_{C^*}$ be the language of unital C*-algebras, let $P_{\min}$ be a unary predicate symbol to be interpreted as the distance to the set of minimal projections, and let $P_{\sim}$ be a binary predicate symbol to be interpreted as the distance to the set of unitarily conjugate pairs (more detailed definitions are in Definition \ref{defn:RhoMin} and Definition \ref{defn:RhoSim}, respectively).  We prove the following:

\begin{iThm}
Every finite dimensional C*-algebra admits quantifier elimination in $L_{C^*} \cup \{P_{\min}, P_{\sim}\}$.
\end{iThm}

Both of the predicates we are adding to the language of unital C*-algebras are definable in every finite dimensional C*-algebra (see Theorem \ref{thm:RhoMinDfbl} and Theorem \ref{thm:RhoSimDfbl}).  Our results (together with those of \cite{Eagle2016} and \cite{Eagle2017}) therefore also imply these predicates are not quantifier-free definable.  Thus $P_{\min}$ and $P_{\sim}$ are the only obstacles to quantifier elimination in the class of finite dimensional C*-algebras.

We also consider to what extent using just one of the two new predicates suffices to obtain quantifier elimination.  In the positive direction, we show:

\begin{iThm}
For every $n \in \mathbb{N}$, the full matrix algebra $M_n(\mathbb{C})$ admits quantifier elimination in $L_{C^*} \cup \{P_{\min}\}$.
\end{iThm}

On the other hand, we produce an example of a finite dimensional C*-algebra which does not have quantifier elimination in $L_{C^*} \cup \{P_{\min}\}$ (Example \ref{min:issues}), and another example that does not have quantifier elimination in $L_{C^*} \cup \{P_{\sim}\}$ (Example \ref{unitary:issues}).  Thus neither adding $P_{\min}$ alone, nor adding $P_{\sim}$ alone, is sufficient to obtain quantifier elimination for all finite dimensional C*-algebras.

The remainder of the paper is organized as follows.  In Section 1 we summarize the results we will need from model theory and C*-algebras.  We assume that the reader is familiar with the basics of both C*-algebras and continuous model theory, so we only present the specific parts of those theories that we will need, and we refer the reader to other sources for a more detailed exposition.  In Section 2 we consider the consequences of adding a predicate for minimal projection to the language of C*-algebras; we show that this gives quantifier elimination for full matrix algebras, but not for all finite dimensional C*-algebras.  In Section 3 we add the predicate for unitarily conjugate pairs, and show that together with the predicate for minimal projections we obtain quantifier elimination for all finite dimensional C*-algebras.  We conclude in Section 4 with some natural questions not resolved in the present work.  Throughout the paper, all C*-algebras are assumed to be unital.

\subsection*{Acknowledgements}
We thank Ilijas Farah, Isaac Goldbring, and Ian Putnam for helpful discussions in the early stages of this project.  The second author was financially supported by a Jamie Cassels Undergraduate Research Award at the University of Victoria.

\section{Model theory and C*-algebras}
We begin with a brief presentation of the aspects of the model theory of C*-algebras that will be useful to us later.  The reader interested in more detail is encouraged to consult \cite{BenYaacov2008a} for continuous logic, \cite{Davidson1996} for C*-algebras, and \cite{Farah} for the model theory of C*-algebras.

\subsection*{C*-algebras}
Throughout this paper we consider only \emph{unital} C*-algebras.  In fact, our main focus is on finite dimensional C*-algebras (which are always unital), but we will make occasional remarks about infinite dimensional algebras as well.  By an \emph{embedding} of one C*-algebra into another, we mean an injective unital *-homomorphism.  We will use the following fact repeatedly, usually without mention.

\begin{fact}[{\cite[Theorem III.1.1]{Davidson1996}}]\label{fact:DirectSum}
Every finite dimensional C*-algebra is isomorphic to a finite direct sum of algebras of the form $M_n(\mathbb{C})$.
\end{fact}

We will also need a characterization of the embeddings between finite dimensional C*-algebras.

\begin{fact}[{\cite[Lemma III.2.1]{Davidson1996}}]\label{fact:Embeddings}
There is an embedding of $M_n(\mathbb{C})$ into $M_k(\mathbb{C})$ if and only if $n$ divides $k$.  Embeddings of one finite dimensional C*-algebra $A$ into another finite dimensional C*-algebra $B$ are classified, up to conjugation by a unitary in $B$, by the multiplicities with which each direct summand of $A$ is mapped into each direct summand of $B$.
\end{fact}

A \emph{Bratteli diagram} is a convenient device for tracking the multiplicity information necessary for classifying an embedding between finite dimensional C*-algebras.  Given algebras $A = \bigoplus M_{n_i}(\mathbb{C})$ and $B = \bigoplus M_{m_j}(\mathbb{C})$, and an embedding $\phi : A \to B$, we let $E_\phi(n_i, m_j)$ be the multiplicity with which $\phi$ embeds $M_{n_i}$ into $M_{m_j}$; visually, we represent the Bratteli diagram as a bipartite multigraph, with the summands of $A$ as one part and the summands of $B$ as the other, and with $E_\phi(n_i, m_j)$ edges drawn between $M_{n_i}$ and $M_{m_j}$.  We define the \emph{multiplicity} of a direct summand of $A$ or $B$ to be the number of edges incident with that summand in the Bratteli diagram.  More formally, for any $i$,
\[\op{mult}_\phi(M_{n_i}(\mathbb{C})) = \sum_j E_\phi(n_i, m_j),\]
and for any $j$,
\[\op{mult}_\phi(M_{m_j}(\mathbb{C})) = \sum_i E_\phi(n_i, m_j).\]
We emphasize that in both cases the sums involved are finite, by Fact \ref{fact:DirectSum}.

An element $v$ of a C*-algebra $A$ is \emph{unitary} if $v^*v = vv^* = 1$.  If $a$ and $b$ are elements of $A$, and there is a unitary $v \in A$ such that $a = v^*bv$, then we call $a$ and $b$ \emph{unitarily conjugate in $A$}, and write $a \sim_u b$.

For each $n \in \mathbb{N}$ we let $0_n$ denote the $n \times n$ matrix with every entry equal to $0$.

\subsection*{Model theory}
We treat C*-algebras model-theoretically through the use of continuous logic.  The appropriate formalism for treating C*-algebras in continuous logic was introduced in \cite{Farah2014a}, and we also refer the reader to \cite{Farah} for a more thorough treatment.

We let $L_{C^*}$ denote the language of unital C*-algebras, and recall that there is a theory $T_{C^*}$ such that if $M$ is an $L_{C^*}$-structure and $M \models T_{C^*}$, then $M$ is a unital C*-algebra.  Atomic formulas of $L_{C^*}$ correspond to expressions of the form $\norm{p(x_1, \ldots, x_n)}$, where $p$ is a *-polynomial with complex coefficients.  An embedding of C*-algebras, as defined above, corresponds exactly to the model-theoretic notion of an $L_{C^*}$-embedding.

We recall that in continuous logic $\sup$ and $\inf$ (over balls of finite radius) play the role of the quantifiers $\forall$ and $\exists$, respectively.  We therefore say that a formula is \emph{quantifier-free} if it can be constructed from the atomic formulas by the application of continuous functions (which play the role of connectives), without using the quantifiers $\sup$ and $\inf$.  A theory $T$ has \emph{quantifier elimination} if every formula can be approximated, uniformly over all tuples in all models of $T$, by a quantifier-free formula.  More precisely:

\begin{defn}\label{defn:QE}
A theory $T$ in a language $L$ has \emph{quantifier elimination} if for every $L$-formula $\phi(\overline{x})$, and every $\epsilon > 0$, there is a quantifier-free $L$-formula $\psi_{\epsilon}(\overline{x})$ such that for every model $M$ of $T$, and every tuple $\overline{a}$ from $A$ of the same length as $\overline{x}$,
\[\abs{\phi^M(\overline{a}) - \psi_\epsilon^M(\overline{a})} \leq \epsilon.\]

If $M$ is an $L$-structure we say $M$ has quantifier elimination if $\op{Th}(M)$ does.
\end{defn}

We emphasize that whether or not a theory $T$ has quantifier elimination depends not only on $T$, but also on the language in which $T$ is expressed.  Our main goal in this paper is to find a suitable expansion of the language of C*-algebras, and a canonical way of interpreting C*-algebras as structures in the expanded language, such that in the expanded language all finite dimensional C*-algebras have quantifier elimination.

A language $L'$ is an \emph{expansion} of the language $L$ if every $L$-formula is also an $L'$-formula.  We will be considering languages that expand the language $L_{C^*}$.  If $L'$ is an expansion of $L_{C^*}$, then by an \emph{$L'$-embedding} of $A$ into $B$ we mean a unital injective *-homomorphism $f : A \to B$ such that, for every $(a_1, \ldots, a_n) \in A^n$, $P^A(a_1, \ldots, a_n) = P^B(f(a_1), \ldots, f(a_n))$.

The main test for quantifier elimination is the following.  In our case, $L$ will always be an expansion of $L_{C^*}$, but the test holds in much more generality.

\begin{thm}[{\cite[Proposition 13.6]{BenYaacov2008a}}]\label{thm:GeneralQETest}
Let $L$ be a language for metric structures and $T$ be an $L$-theory.  The following are equivalent:
\begin{enumerate}
\item{$T$ has quantifier elimination.}
\item{Whenever $A \models T$ and $B \models T$, every embedding of a substructure of $A$ can be extended to an embedding of $A$ into an ultrapower of $B$.}
\end{enumerate}
\end{thm}

Working with finite dimensional C*-algebras provides a considerable simplification to the above theorem, owing to the following fact.  This fact, in particular, explains why we never need to work with ultrapowers in this paper.

\begin{fact}\label{fact:FDEquivalence}
Let $A$ be a C*-algebra.  The following are equivalent:
\begin{itemize}
\item{$A$ is finite dimensional,}
\item{the unit ball of $A$ is compact,}
\item{every ultrapower of $A$ is isomorphic to $A$}
\item{up to $L$-isomorphism, $A$ is the only model of $\op{Th}_L(A)$, where $L$ is any language expanding $L_{C^*}$.}
\end{itemize}
\end{fact}

Combining Theorem \ref{thm:GeneralQETest} with Fact \ref{fact:FDEquivalence} gives the following test for quantifier elimination, which will be our main tool.

\begin{thm}\label{thm:MainQETest}
Let $A$ be a finite dimensional C*-algebra, and let $L$ be a language expanding the language $L_{C^*}$ of C*-algebras.  The following are equivalent:
\begin{enumerate}
\item{
The $L$-theory $\op{Th}_L(A)$ has quantifier elimination,}
\item{for any $L$-substructure $C$ of $A$, and any two $L$-embeddings $\phi:C \to A$ and $\psi:C \to A$, there exists an $L$-automorphism $\theta:A \to A$ such that $\theta \circ \psi = \phi$; that is, the following diagram commutes.
	\begin{equation*}\label{fin dim amalgamation diagram}
	\begin{tikzcd}[ampersand replacement = \&]
	A \arrow[r,"\theta"] \& A\\
	C \arrow[u, "\phi"] \arrow[ur, "\psi"']
	\end{tikzcd}
	\end{equation*}
}
\end{enumerate}
\end{thm}

Theorem \ref{thm:MainQETest} relies heavily on the fact that a finite dimensional C*-algebra is the only model of its theory.  This produces significant difficulties in the infinite dimensional case (see Section \ref{sec:Questions}).

\section{Minimal projections}\label{sec:Min}
In this section we consider the consequences of adding a new symbol to the language of C*-algebras, representing the distance to the set of minimal projections.

\begin{defn}
Let $A$ be a C*-algebra.  A \emph{projection} in $A$ is an element $p \in A$ such that $p^2=p^*=p$.  A nonzero projection $p$ is \emph{minimal} if whenever $q$ is another projection in $A$, and $qp=q$, then $q=0$.
\end{defn}

If $A = B(H)$, then a minimal projection is precisely the orthogonal projection onto a subspace of $H$ of dimension $1$; that is, $p$ is minimal if and only if $pB(H)p \cong \mathbb{C}$.  In general C*-algebras the condition $pAp \cong \mathbb{C}$ is strictly stronger than $p$ being minimal, but in finite dimensional C*-algebras (which is the case we are concerned with here), these two notions are equivalent.  The following lemma follows directly from Fact \ref{fact:DirectSum}.

\begin{lem}\label{lem:Min:MinimalRank1}
Let $A$ be a finite dimensional C*-algebra, and let $p$ be a projection in $A$.  Then $pAp \cong \mathbb{C}$ if and only if $p$ is a minimal projection in $A$.
\end{lem}

\begin{defn}\label{defn:RhoMin}
We denote by $\rho_{\op{min}}$ the predicate measuring the distance to the set of minimal projections.  That is, in any C*-algebra $A$,
\[\rho_{\op{min}}^A(x) = \inf\{\norm{x - p} : p \in A, p \text{ is a minimal projection}\}.\]
\end{defn}

\begin{thm}\label{thm:RhoMinDfbl}
The predicate $\rho_{\op{min}}$ is definable in each finite dimensional C*-algebra.
\end{thm}
\begin{proof}
In any finite dimensional C*-algebra $A$, $\rho_{\op{min}}^A$ is exactly the distance to the set of rank 1 projections, by Lemma \ref{lem:Min:MinimalRank1}.  The distance to the set of rank 1 projections is definable by \cite[Lemma 3.13.4]{Farah}, because the property of a C*-algebra being isomorphic to $\mathbb{C}$ is elementary and co-elementary.
\end{proof}

We let $L_{\min} = L_{C^*} \cup \{P_{\min}\}$, where $P_{\min}$ is a new unary predicate symbol. 

We next show that, for the class of full matrix algebras, $\rho_{\op{min}}$ is not quantifier-free definable, and is, in a sense, the only obstacle to quantifier elimination for such algebras.

\begin{lem}\label{lem:Min:Restriction}
Let $A = \bigoplus_{j=1}^k M_{n_j}(\mathbb{C})$ and $B = \bigoplus_{j=1}^l M_{m_j}(\mathbb{C})$ be matrix algebras, and let $\phi:A \to B$ be a injective unital $^*$-homomorphism.  If $\phi$ is an $L_{\min}$-embedding, then $\mult_\phi(M_{n_j}(\mathbb{C})) = 1$ for each $j$.
\end{lem}
\begin{proof}
If $\phi$ is an $L_{\min}$-embedding, then $\phi$ preserves rank:
\begin{equation}
e_{j} = 0_{n_1} \oplus \cdots \oplus 0_{n_{j-1}} \oplus \begin{bmatrix}1 & & \\ & \ddots & \\ & & 0\end{bmatrix} \oplus 0_{n_{j+1}} \oplus\cdots 0_{n_k} \in A
\end{equation}
and similarly 
\begin{equation}\label{f thing}
f_{j} = 0_{m_1} \oplus \cdots \oplus 0_{m_{j-1}} \oplus \begin{bmatrix}1 & & \\ & \ddots & \\ & & 0\end{bmatrix} \oplus 0_{m_{j+1}} \oplus\cdots 0_{m_l} \in B.
\end{equation}
A projection in $A$ is minimal if and only if it is unitarily conjugate in $A$ to $f_i$ for some $i$, and similarly a projection in $B$ is minimal if and only if it is unitarily conjugate in $B$ to some $f_j$. Since $\phi$ is a unital $^*$-homomorphism, $\phi$ takes unitarily conjugate pairs to unitarily conjugate pairs. Hence, for each minimal projection $p \in A$, \[
p \sim_u e_j \Longrightarrow \phi(p) \sim_u f_{i}
\] 
for some $i$ and $j$. If $\mult_\phi(M_{n_j}(\mathbb{C})) > 1$, then $\phi(p)$ is the sum of multiple mutually orthogonal projections and therefore has rank greater than $1$. Since $f_i$ has rank $1$, the implication above tells us that $\mult_\phi(M_{n_j}(\mathbb{C})) = 1$.
\end{proof}

\begin{thm}\label{thm:Min:MatrixAlgebras}
For every $n \in \mathbb{N}$, the $L_{\op{min}}$-theory of $M_n(\mathbb{C})$ has quantifier elimination.
\end{thm}
\begin{proof}
Let $C$ be a sub-$C^*$-algebra of $M_n(\mathbb{C})$ and let $\iota : C \into M_n(\mathbb{C})$ be an $L_{\min}$-embedding. $C$ is a finite dimensional $C^*$-algebra, so there exist natural numbers $n_1, \ldots,n_k$ such that\[
n_1 + \cdots + n_k \leq n,
\] 
as well as an isomorphism \[
\phi:C \cong M_{n_1}(\mathbb{C})\oplus\cdots\oplus M_{n_k}(\mathbb{C}). 
\]
Identifying $C$ with $\phi(C)$, by Lemma \ref{lem:Min:Restriction} the Bratteli diagram corresponding to $\iota$ is of the form\[
\begin{tikzcd}
 & M_n(\mathbb C) &  \\
M_{n_1}(\mathbb C) \arrow[r, "\oplus" description, no head, dotted] \arrow[ru, no head] & \cdots \arrow[r, "\oplus" description, no head, dotted] & M_{n_k}(\mathbb C) \arrow[lu, no head]
\end{tikzcd}
\]
Since this form is independent of $\iota$, it follows from fact \ref{fact:Embeddings} that any two $L_{\min}$-embeddings into $M_n(\mathbb{C})$ are unitarily conjugate in $M_n(\mathbb{C})$, and hence are amalgamated by an inner automorphism of $M_n(\mathbb{C})$.
\end{proof}

One might hope that Theorem \ref{thm:Min:MatrixAlgebras} holds for arbitrary finite dimensional C*-algebras, not just those of the form $M_n(\mathbb{C})$.  The next example shows that this is not the case.

\begin{ex}\label{min:issues}
Both of the embeddings depicted in the following diagram are $L_{\min}$ embeddings.
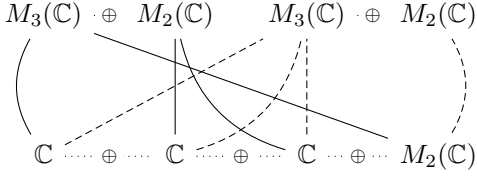
\begin{equation*}
\begin{tikzcd}[ampersand replacement=\&, column sep = small]
M_3(\mathbb C) \arrow[r, "\oplus" description, no head, dotted] \& M_2(\mathbb C) \& M_3(\mathbb C) \arrow[r, "\oplus" description, no head, dotted] \& M_2(\mathbb C) \\
 \&  \&  \&  \\
\mathbb C \arrow[r, "\oplus" description, no head, dotted] \arrow[rruu, no head, dashed] \arrow[uu, no head, bend left] \& \mathbb C \arrow[r, "\oplus" description, no head, dotted] \arrow[ruu, no head, dashed, bend right] \arrow[uu, no head] \& \mathbb C \arrow[r, "\oplus" description, no head, dotted] \arrow[uu, no head, dashed] \arrow[luu, no head, bend left] \& M_2(\mathbb C) \arrow[uu, no head, dashed, bend right] \arrow[llluu, no head]
\end{tikzcd}
\end{equation*}
These embeddings cannot be amalgamated, so by Theorem \ref{thm:MainQETest} the $L_{\min}$ theory of $\mathbb{C}^3 \oplus M_2(\mathbb{C})$ does not have quantifier elimination.
\end{ex}

In the following section we will see that the embeddings in this example can be prohibited by forcing our embeddings to preserve unitary conjugacy.

To conclude this section we make some remarks about the trace on a full matrix algebra.  Let $L_{\op{trace}}$ be the language of unital C*-algebras, together with a new predicate symbol (intended to represent a trace on the algebra).  It was already noticed in \cite{Eagle2017} that full matrix algebras have quantifier elimination in $L_{\op{trace}}$, when the new symbol is interpreted as the unique trace on the matrix algebra.  This observation about the trace and our Theorem \ref{thm:Min:MatrixAlgebras} are essentially equivalent.  The reason is that all matrix algebras have real rank $0$, and thus the trace is completely determined by its behaviour on projections.  Moreover, for any projection $p \in M_n(\mathbb{C})$, we have $\op{tr}(p) = \op{rank}(p)/n$.  Thus the minimal projections in $M_n(\mathbb{C})$ are exactly the projections with trace $\frac{1}{n}$.  Conversely, in $M_n(\mathbb{C})$ every projection $p$ is a sum of minimal projections, and the rank of $p$ is the minimal number of minimal projections needed to obtain $p$.  Thus, in the case of full matrix algebras, the trace and $\rho_{\min}$ are quantifier-free interdefinable.

Finite dimensional C*-algebras other than full matrix algebras do not have a unique trace.  Therefore if $A$ is a finite dimensional C*-algebra that is not a full matrix algebra, then there is not a canonical way to view $A$ as an $L_{\op{trace}}$ structure, while there is a canonical way to view it as an $L_{\min}$ structure.  It is for this reason that we have chosen a treatment based on minimal projections rather than traces.

\section{Unitarily conjugate pairs}\label{sec:Unitary}
Having seen that the distance to minimal projections does not suffice to obtain quantifier elimination in all finite dimensional C*-algebras, we introduce another new predicate.

\begin{defn}\label{defn:RhoSim}
Let $A$ be a unital C*-algebra.  We denote by $\rho_\sim^A$ the predicate measuring the distance (in $A^2$) to the set of unitarily conjugate pairs in $A^2$.  That is, if we define
\[S = \{(a, b) \in A : u^*au = b \text{ for some unitary } u \in A\},\]
then for any $x, y \in A$,
\[\rho_\sim(x, y) = \inf\{\max\{\norm{x-a}, \norm{y-b}\} : (a, b) \in S\}.\]
\end{defn}

\begin{thm}\label{thm:RhoSimDfbl}
The predicate $\rho_\sim$ is $L_{C^*}$-definable, uniformly in all C*-algebras.
\end{thm}
\begin{proof}
For any C*-algebra $A$, let $U(A)$ denote the set of unitaries in $A$.  It is shown in \cite[Example 3.2.7]{Farah} that $U(A)$ is a definable set.  It follows (see \cite[Definition 3.2.3]{Farah}) that the expresssion $\psi(x, y) = \inf_{u \in U(A)} \norm{u^*xu - y}$ is a definable predicate.  

It is immediate from the definitions that $\rho_\sim^A$ is exactly the distance to the zero set of $\psi^A$.  By \cite[Lemma 3.2.5]{Farah}, to complete our proof it is sufficient to show that $\psi$ is \emph{weakly stable}, that is, we must show that for every $\epsilon > 0$ there exists $\delta > 0$ such that in every C*-algebra $A$, if $(x, y) \in A^2$ satisfies $\psi^A(x, y) < \delta$ then there exists $(a, b) \in A^2$ such that $\max\{\norm{x-a}, \norm{y-b}\} < \epsilon$ and $\psi^A(a, b) = 0$.

Given $\epsilon > 0$, pick $\delta = \epsilon$.  If $\psi^A(x, y) < \epsilon$, then $\inf_{u \in U(A)} \norm{u^*xu - y} < \epsilon$, and hence there is some $u \in U(A)$ such that $\norm{u^*xu - y} < \epsilon$.  Let $a = x$ and $b = u^*xu$.  Then $\psi^A(a, b) = 0$, and $\max\{\norm{x-a}, \norm{y-b}\} = \max\{0, \norm{y-u^*xu}\} < \epsilon$.
\end{proof}

Let $L_* = L_{C^*} \cup \{P_{\min}, P_{\sim}\}$, where $P_{\min}$ is as defined in Section \ref{sec:Min}, and $P_{\sim}$ is a new binary predicate symbol.

\begin{lem}\label{lem:Sim:Restriction}
Let $A = \bigoplus_{j=1}^k M_{n_j}(\mathbb{C})$ and $B = \bigoplus_{j=1}^l M_{m_j}(\mathbb{C})$ be finite dimensional C*-algebras, and let $\phi : A \to B$ be an $L_*$-embedding.  Then $\op{mult}_\phi(M_{m_j}(\mathbb{C})) = 1$ for all $j$.
\end{lem}
\begin{proof}
Suppose, for a contradiction, that there is some $j$ for which $\op{mult}_\phi(M_{m_j}(\mathbb{C})) > 1$.  By Lemma \ref{lem:Min:Restriction} the multiplicity of each summand of $A$ is $1$, so in order to have $\op{mult}_\phi(M_{m_j}(\mathbb{C})) > 1$ there must be two distinct direct summands $M_{n_\ell}(\mathbb{C})$ and $M_{n_{\ell'}}(\mathbb{C})$ of $A$ such that $\phi$ embeds a copy of each of $M_{n_\ell}(\mathbb{C})$ and $M_{n_{\ell'}}(\mathbb{C})$ into $M_{m_j}(\mathbb{C})$, i.e., $E_\phi(n_\ell, m_j) = 1 = E_\phi(n_{\ell'}, m_j)$.  Since $\op{mult}_\phi(M_{n_\ell}(\mathbb{C}))=1=\op{mult}_\phi(M_{n_{\ell'}}(\mathbb{C}))$, we may restrict the domain and hence assume that $A = M_{n_\ell}(\mathbb{C}) \oplus M_{n_{\ell'}}(\mathbb{C})$ and $B = M_{m_j}(\mathbb{C})$.

Let $p$ be a minimal projection in $M_{n_\ell}(\mathbb{C})$, and let $q$ be a minimal projection in $M_{n_{\ell'}}(\mathbb{C})$.  Since $M_{n_\ell}(\mathbb{C})$ and $M_{n_{\ell'}}(\mathbb{C})$ are distinct direct summands of $A$, we have that $p \oplus 0$ and $0 \oplus q$ are not unitarily conjugate in $A$.  On the other hand, since $\phi$ is an $L_{\min}$-embedding, $\phi(p \oplus 0)$ and $\phi(0 \oplus q)$ are minimal projections in $B = M_{m_j}(\mathbb{C})$,  and hence are unitarily equivalent in $B$.  This contradicts that $\phi$ is an $L_{\sim}$-embedding.
\end{proof}

\begin{thm}\label{thm:LStar}
Let $A$ be a finite dimensional C*-algebra, viewed as an $L_*$ structure by interpreting $P_{\min}$ as $\rho_{\min}^A$ and $P_\sim$ as $\rho_{\sim}^A$.  Then $\op{Th}_{L_*}(A)$ has quantifier elimination.
\end{thm}
\begin{proof}
We aim to appeal to Theorem \ref{thm:MainQETest}. To this end, find $k \in \mathbb{N}$, and integers $n_1,\cdots,n_k \in \mathbb{N}$, such that $A \cong \bigoplus_{j=1}^k M_{n_j}(\mathbb{C})$. Let $C$ be an $L_*$-substructure of $A$. Since $C$ is a sub-$C^*$-algebra of $A$, $C$ is finite-dimensional. Again, find $l \in \mathbb{N}$ and $m_1,\cdots,m_l \in \mathbb{N}$ such that $C \cong \bigoplus_{j=1}^l M_{m_j}(\mathbb{C})$.  Let $\iota : C \into A$ be the inclusion map, and let $\phi : C \into A$ be another $L_*$-embedding.

Lemma \ref{lem:Min:Restriction} tells us that $\mult_\phi(M_{m_j}(\mathbb{C})) = 1$ for each $j$, and Lemma \ref{lem:Sim:Restriction} tells us that $\mult_\phi(M_{n_j}(\mathbb{C})) = 1$ for each $j$.  Hence, since $\phi$ is unital, if $E_\phi(n_i, m_j) > 0$, then $n_i = m_j$, and $E_\phi(n_i, m_j) = 1$.  Thus $\phi$ is an isomorphism, and hence $A$ and $C$ have the same summands in their direct sum decomposition. Therefore, to amalgamate the diagram
	\begin{equation*}
	\begin{tikzcd}[ampersand replacement = \&]
	A \arrow[r,"\theta"] \& A\\
	C \arrow[u, "\phi"] \arrow[ur, "\subseteq"']
	\end{tikzcd}
	\end{equation*}
we may choose $\theta$ to simply reorder the direct summands of $A$ so as to match the order of those summands in $C$.  By Theorem \ref{thm:MainQETest}, $\op{Th}_{L_*}(A)$ admits quantifier elimination.
\end{proof}

\begin{rem}
The proof of Theorem \ref{thm:LStar} also shows that every $L_*$-embedding of one finite dimensional C*-algebra into another is actually an isomorphism.  This phenomenon is unavoidable, in the following sense.  Suppose that $L$ is any expansion of $L_{C^*}$ such that every new symbol in $L$ represents a predicate that is $L_{C^*}$-definable in all finite dimensional C*-algebras, and that all finite dimensional C*-algebras have quantifier elimination when viewed as $L$-structures.  Since $\rho_{\min}$ and $\rho_{\sim}$ are $L_{C^*}$-definable in every finite dimensional C*-algebra they are also $L$-definable in every finite dimensional C*-algebra, and by quantifier elimination in $L$ they are quantifier-free $L$-definable.  Therefore every $L$-embedding must preserve the values of $\rho_{\min}$ and $\rho_{\sim}$, and so the proof of Theorem \ref{thm:LStar} implies that every $L$-embedding is also an $L$-isomorphism.
\end{rem}

As was the case with $P_{\min}$, adding only $P_{\sim}$ to $L_{C^*}$ is insufficient for obtaining quantifier elimination.

\begin{ex}\label{unitary:issues}
Both of the embeddings in the following diagram are $L_{\sim}$ embeddings.
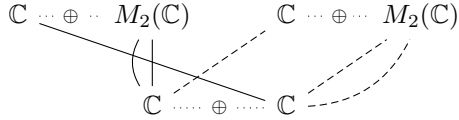
\[
\begin{tikzcd}[ampersand replacement = \&]
\mathbb C \arrow[r, "\oplus" description, no head, dotted] \& M_2(\mathbb C) \arrow[d, no head, bend right] \& \mathbb C \arrow[r, "\oplus" description, no head, dotted] \& M_2(\mathbb C) \\
 \& \mathbb C \arrow[r, "\oplus" description, no head, dotted] \arrow[u, no head] \arrow[ru, no head, dashed] \& \mathbb C \arrow[llu, no head] \arrow[ru, no head, dashed] \arrow[ru, no head, dashed, bend right] \& 
\end{tikzcd}
\]
The diagram above cannot be amalgamated, as one of the two embeddings sends $1 \oplus 0$ to $1 \oplus 0_2$ and the other sends $1 \oplus 0$ to $1_2 \oplus 0$. Since they differ in rank, $1 \oplus 0_2 \not\sim_u 1_2 \oplus 0$.
\end{ex}

\section{Questions}\label{sec:Questions}
The main results of this paper show that every finite dimensional C*-algebra admits quantifier elimination in $L_*$, and that in any finite dimensional C*-algebra $L_*$-definability coincides with $L_{C^*}$-definability.   Our proof also shows that the definition of $\rho_{\sim}$ is uniform across the class of all C*-algebras.  On the other hand, our proof $\rho_{\min}$ is definable in each finite dimensional C*-algebra does not guarantee that the definition of $\rho_{\min}$ is the same in each algebra.   Since the class of finite dimensional C*-algebras is not an elementary class, the usual techniques for showing that a predicate is uniformly definable in a class of structures do not directly apply.

\begin{q}
Is $\rho_{\min}$ uniformly definable in the class of finite dimensional C*-algebras?
\end{q}

Our proofs depend heavily on Theorem \ref{fin dim amalgamation diagram}, which in turn relies on the fact that if $A$ is finite dimensional, then $A$ is isomorphic to each of its ultrapowers.  If $A$ is infinite dimensional then $A$ has ultrapowers that are not isomorphic to $A$ itself.  In the infinite dimensional context one therefore has to use Theorem \ref{thm:GeneralQETest} directly, and the techniques used in this paper do not apply.  The most optimistic version of the question is the following:

\begin{q}\label{q2}
Is there a finite collection $P_1, \ldots, P_n$ of predicates, each of which is uniformly definable in the class of unital C*-algebras, such that every unital C*-algebra has quantifier elimination in $L_{C^*} \cup \{P_1, \ldots, P_n\}$?
\end{q}

\begin{rem}
It is not difficult to see that $\rho_{\min}$ and $\rho_{\sim}$ alone are not sufficient for obtaining quantifier elimination in the class of all unital C*-algebras.  For instance, consider the algebra $C([0, 1])$ of continuous functions on the unit interval.  In this algebra the only projections are $0$ and $1$, both of which appear as constant symbols in $L_{C^*}$, so $\rho_{\min}^{C([0, 1])}$ is quantifier-free $L_{C^*}$-definable in $C([0, 1])$.  On the other hand, $C([0, 1])$ is commutative, so in this algebra $\rho_\sim$ is just the distance to the diagonal in $C([0, 1])^2$, which is again quantifier-free $L_{C^*}$-definable.  Thus $C([0, 1])$, viewed as an $L_*$-structure, still does not have quantifier elimination.
\end{rem}

One can weaken Question \ref{q2} by only asking that the predicates be definable in each completion of the theory of unital C*-algebras, rather than being uniformly definable.  One might expect to make progress first for UHF algebras, i.e., those infinite dimensional separable C*-algebras which are the direct limits of sequences of full matrix algebras.  Using Theorem \ref{thm:GeneralQETest} still appears difficult, however, because subalgebras of UHF algebras are not always UHF (or even nuclear, see \cite{Blackadar1985}), and ultrapowers of UHF algebras are almost never UHF (see \cite{Carlson2014}).  For this reason, we do not even know how to answer the following special case of Question \ref{q2}:

\begin{q}
Let $A$ be the CAR algebra, that is, the direct limit of the sequence $M_2(\mathbb{C}) \hookrightarrow M_4(\mathbb{C}) \hookrightarrow \cdots \hookrightarrow M_{2^n}(\mathbb{C}) \hookrightarrow \cdots$.
Is there a finite collection $P_1, \ldots, P_n$ of predicates, each of which is definable in the $L_{C^*}$-theory of $A$, such that $A$ has quantifier elimination in $L_{C^*} \cup \{P_1, \ldots, P_n\}$?
\end{q}

% ----------------------------------------------------------------
\bibliographystyle{amsalpha}
\bibliography{QEBarrier}
\end{document}